\documentclass[11pt,a4paper,twoside]{article}
\usepackage[margin=25mm]{geometry}
\usepackage{amsmath,amssymb,amsthm}
\usepackage{enumitem,xcolor,hyperref}
\usepackage{fancyhdr}
\makeatletter
\AtBeginDocument{\let\runningauthor\@author}
\renewcommand{\section}{\@startsection{section}{1}{\z@}
  {-3.5ex \@plus -1ex \@minus -.2ex}
  {2.3ex \@plus .2ex}
  {\normalfont\normalsize\scshape\centering}}
\renewcommand{\@seccntformat}[1]{\csname the#1\endcsname.\quad}
\makeatother
\hypersetup{colorlinks=true,linkcolor=blue!45!black,citecolor=green!35!black,
urlcolor=blue!55!black,pdftitle={Index one in a half-space: the free boundary modifications}}
\setlist{nosep}
\numberwithin{equation}{section}
\newtheorem{theorem}{Theorem}[section]
\newtheorem{lemma}[theorem]{Lemma}
\newtheorem{proposition}[theorem]{Proposition}

\theoremstyle{definition}
\newtheorem*{definition}{Definition}
\theoremstyle{remark}
\newtheorem{remark}[theorem]{Remark}
\newcommand{\R}{\mathbb R}
\newcommand{\dd}{\mathrm d}
\newcommand{\Q}{\mathcal Q}
\newcommand{\BN}{\mathcal B_N}
\newcommand{\EE}{\mathcal E_{\mathrm{rel}}}
\newcommand{\ip}[2]{\langle #1,#2\rangle}
\DeclareMathOperator{\Ind}{Ind}
\DeclareMathOperator{\Span}{span}

\title{FREE BOUNDARY MINIMAL HYPERSURFACES OF MORSE INDEX ONE IN EUCLIDEAN HALF-SPACE}
\date{}
\author{Junzhang Li}
\begin{document}
\maketitle
\begin{abstract}
We prove that a complete, connected, two-sided, embedded free
boundary minimal hypersurface of finite total curvature and Morse index
one in a half-space of $\R^N$, with nonempty boundary is a higher-dimensional
half-catenoid. We also prove that a complete, connected, two-sided, embedded free
boundary minimal hypersurface of finite total curvature and Morse index
one in a quarter-space is the higher-dimensional quarter-catenoid.
\end{abstract}

\section{Introduction}\label{sec:setup}

Let $N=n+1\ge4$.  Write $x=(x_1,\ldots,x_N)$ for the standard
coordinates on $\R^N$, let $e_1,\ldots,e_N$ be the corresponding
orthonormal basis, and set
\[
 H=\{x\in\R^N:x_N\ge0\},\qquad
 P=\partial H=\{x_N=0\},\qquad z=x_N.
\]
We say that a minimal hypersurface $\Sigma^n\subset H$ has free boundary
on $P$ if $\partial\Sigma=\Sigma\cap P\ne\varnothing$ and $\Sigma$ meets
$P$ orthogonally along $\partial\Sigma$. For $f\in C_c^\infty(\Sigma)$, the second
variation of area is
\begin{equation}\label{eq:intro-Q}
\Q(f,f)=\int_\Sigma\bigl(|\nabla f|^2-|A|^2f^2\bigr)\,\dd\mu,
\end{equation}
where $A$ is the second fundamental form. Set
\[
 I_N(\Sigma):=\sup\bigl\{\dim U:U\subset C_c^\infty(\Sigma),\ 
 \Q(f,f)<0\ \text{for every }f\in U\setminus\{0\}\bigr\}
\]
and
\[
 I_D(\Sigma):=\sup\bigl\{\dim U:U\subset C_c^\infty(\Sigma^\circ),\ 
 \Q(f,f)<0\ \text{for every }f\in U\setminus\{0\}\bigr\},
\]
and call $I_N(\Sigma)$ and $I_D(\Sigma)$  Neumann Morse index and Dirichlet Morse index respectively.
The difference is that the latter one only allows variations fixing $\partial\Sigma$.
We say that $\Sigma$ has finite total curvature if
$\int_\Sigma|A|^n\,\dd\mu<\infty$.

Ambrozio--Buzano--Carlotto--Sharp \cite{ABCS} studied the Neumann and
Dirichlet indices of free boundary minimal hypersurfaces in a half-space
and their relation to the index of the reflected double. In $\R^3$,
they proved that an index-one half-bubble is isometric to a half-catenoid
\cite[Corollary 24]{ABCS}. Here a half-bubble is complete, connected,
properly embedded, has nonempty free boundary in a half-space, and has
finite total curvature. Shuli Chen \cite{Chen} subsequently sharpened
the index comparison using Dirichlet nullity and ruled out index-two
half-bubbles in $\R^3$.

Recently, Chodosh--Gianocca \cite{CG} proved that every complete, connected, embedded
minimal hypersurface of finite total curvature and Morse index one in
$\R^N$ is a higher-dimensional catenoid.  Their proof enlarges the usual space of
$L^2$ harmonic one-forms by allowing prescribed asymptotic behavior and
uses admissible pairings of ambient two-forms.  Building on that method,
we prove the corresponding free boundary classification in a half-space.

\begin{theorem}\label{thm:main}
Let $\Sigma^n\subset H$ be a complete, connected, two-sided, embedded
free boundary minimal hypersurface with finite total curvature.
If $I_N(\Sigma)=1$, then $\Sigma$ is a half-catenoid.
\end{theorem}

Here a half-catenoid is obtained by cutting a higher-dimensional catenoid
along its waist hyperplane or a hyperplane containing its axis.
Both cuts have free boundary and index one. Indeed, their double is
the full catenoid $C$, which has index one by \cite{TZ}.
By \eqref{eq:parity},
\[
1=\Ind(C)=I_N(\Sigma)+I_D(\Sigma),\qquad
0\le I_D(\Sigma)\le I_N(\Sigma).
\]
As the indices are integers, $I_N(\Sigma)=1$ and $I_D(\Sigma)=0$.
Thus the theorem extends
the classification in \cite[Corollary 24]{ABCS} to higher dimensions
under the finite-total-curvature assumption.

\begin{remark}
While this manuscript was being completed, I was made aware of the
independent work of Chenghang Lu \cite{Lu}, who obtained an index-one
classification for capillary minimal hypersurfaces in a half-space
under Euclidean volume growth and certain dimension and contact-angle
restrictions. In particular, \cite[Theorem 2]{Lu} gives the same
half-catenoid classification in the free boundary case for $3\le n\le6$. My proof is obtained simultaneously
and independently
\end{remark}

We briefly describe the proof. Regularity at infinity, embeddedness,
and reflection symmetry imply that the common asymptotic hyperplane of
the ends of the double is either parallel or orthogonal to $P$. In the
orthogonal case, the normal
component of translation perpendicular to $P$ is a decaying Jacobi
field. It can therefore be used in the index form to prove $I_D=0$.
The double then has index one, and \cite[Theorem 1.1]{CG} applies.

In the parallel case, we could not find such a decaying Jacobi field, so we adapt the harmonic 1-form argument of
\cite{CG} to the free boundary setting. If $\Sigma$ has $k$ ends, we
construct a $(k+n)$-dimensional space of exact relative harmonic forms.
The boundary terms in the trace identity vanish, while equality in the
index estimate gives additional Neumann equations. These equations
characterize the equality space as the span of $\dd z$, replacing the 
nullity argument in \cite[Section 8]{CG} by considering the behavior of harmonic forms at infinity. 
The resulting dimension argument forces $k=1$. Then the double has two ends and
is a catenoid by Schoen's theorem \cite[Theorem 3]{Schoen}.

Using the same observation as in the orthogonal case, we can also obtain a similiar classification result in a quarter-space of $R^N$.
For $N=3$, Chen \cite[Proposition 3.1]{Chen} classified quarter-bubbles with index zero and one.
\paragraph{Organization.}
Section~\ref{sec:forms} collects the preliminaries used in the paper.
Section~\ref{sec:orthogonal} treats orthogonal ends. In
Section~\ref{sec:relative} we construct the relative harmonic forms
and compute their dimensions. Section~\ref{sec:trace} establishes the
trace identity and its equality conditions. 
Section~\ref{sec:rigidity} proves the boundary rigidity and completes
the proof of Theorem~\ref{thm:main}. Finally, Section~\ref{sec:quarter} deals with free boundary minimal hypersurfaces in a quarter of $\R^N$.

\paragraph{Use of AI.}
These results were obtained by combining ideas of the author
with a publicly available LLM. The author directed this use, selected and substantially revised any generated text, and independently verified the mathematical statements and proofs. The author takes full responsibility for the final content of the paper.

\paragraph{Acknowledgement}
The author is grateful to his advisor Qiongling Li for advising and encouragemnt. The author would also like to thank Otis Chodosh for suggesting considering the case of the quarter-space, and Yiyang Xiao for helpful discussions
\section{Conventions and preliminaries}\label{sec:forms}

Throughout, $\Sigma$ satisfies the hypotheses of Theorem~\ref{thm:main}.
We use $D_X\nu=-AX$, $\Delta=\operatorname{div}\nabla$, and
$L=\Delta+|A|^2$. Write $B=\partial\Sigma$ and let $\eta$ be its
outward conormal.

Reflection of $\Sigma$ across $P$ gives a smooth minimal double $M$ without boundary, which is smooth by standard elliptic
theory. The double is clearly complete, embedded, two-sided and of finite total curvature. 
We recall the definition of regularity at infinity.
\begin{definition}
A complete minimal hypersurface $M^n\subset\R^{n+1}$, $n\ge3$,
is \emph{regular at infinity} if there is a compact set $K\subset M$
such that
\[
M\setminus K=E_1\sqcup\cdots\sqcup E_s,
\]
where each $E_\alpha$ is a connected graph of a function $v_\alpha$
over the exterior of a bounded region in a hyperplane $\Pi_\alpha$.
In orthonormal coordinates $y=(y_1,\ldots,y_n)$ on $\Pi_\alpha$,
there are constants $h_\alpha,a_\alpha,c_{\alpha,j}\in\R$ such that
\[
D^\ell\left(v_\alpha(y)-h_\alpha-a_\alpha|y|^{2-n}
-\sum_{j=1}^n c_{\alpha,j}y_j|y|^{-n}\right)
=O(|y|^{-n-\ell})
\]
as $|y|\to\infty$, for every integer $\ell\ge0$.
Here $D$ denotes Euclidean differentiation on $\Pi_\alpha$.
The sets $E_\alpha$ represent the ends of $M$; enlarging $K$
only truncates these representatives farther out.
\end{definition}
By the results of Anderson and Schoen \cite{Anderson,Schoen}, the double
$M$ is regular at infinity; see \cite[Theorem A.4]{CG} for the above
formulation and derivative estimates.
Embeddedness makes all end hyperplanes parallel; denote their common direction by $\Pi$.
Reflection preserves the normal line $\Pi^\perp$. Its invariant
lines lie in $P$ or in $\R e_N$, and therefore all ends are either orthogonal or parallel to $P$, denoted as
\begin{equation}\label{eq:end-dichotomy}
\Pi\perp P\qquad\text{or}\qquad\Pi\parallel P.
\end{equation}

Along $B$, the free boundary condition gives $\eta=-e_N$. Taking
components of $D_X\eta=0$ yields
\begin{equation}\label{eq:geometry}
\nabla_X\eta=0,\qquad A(X,\eta)=0\quad(X\in TB),\qquad
A\eta=\kappa\eta,\qquad \kappa=A(\eta,\eta).
\end{equation}
Here we abuse notation by using $A$ to represent both the second fundamental form and the shape operator.
In particular $B$ is totally geodesic in $\Sigma$, and on $B$
\begin{equation}\label{eq:height-boundary}
z=0,\qquad \nabla z=-\eta,\qquad \partial_\eta z=-1.
\end{equation}

The supporting hyperplane $P$ is totally geodesic in $\R^N$, so the free boundary term
$-\int_B\mathrm{II}_P(\nu,\nu)fg$ vanishes, and
\[
\Q(f,g)=\int_\Sigma
\bigl(\ip{\nabla f}{\nabla g}-|A|^2fg\bigr).
\]
Following the energy-space construction in \cite[Section 3.2]{CG}, set
\[
\|f\|_{\mathrm{en}}^2=\int_\Sigma(|\nabla f|^2+|A|^2f^2),\qquad
\BN=\overline{C_c^\infty(\Sigma)}^{\|\cdot\|_{\mathrm{en}}}.
\] Nonflatness gives local $W^{1,2}$
control by the energy norm, i.e. $\BN\subset W^{1,2}_{loc}(\Sigma)$.
Moreover, $|\Q(f,g)|\le\|f\|_{\mathrm{en}}\|g\|_{\mathrm{en}}$;
hence $\Q$ extends continuously to $\BN$, and density of $C^\infty_c(\Sigma)$ in $\BN$ shows that the index remains $I_N$ with respect to functions in $\BN$.
For finite-energy functions not necessarily in $\BN$, we denote the
same integral by $\Q_\infty$.

The index comparison and reflection decomposition of
\cite[Lemmas 16, 18, 20]{ABCS} give
\begin{equation}\label{eq:parity}
I_D(\Sigma)\le I_N(\Sigma),\qquad
\Ind(M)=I_N(\Sigma)+I_D(\Sigma).
\end{equation}
Applying these statements to higher dimensions requires $M$ to have Euclidean volume growth, which follows from regularity at infinity.

\begin{lemma}\label{lem:energy-membership}
If $f\in C^\infty(\Sigma)$ satisfies $f=O(r^{2-n})$ and
$\nabla f=O(r^{1-n})$, then $f\in\BN$.
\end{lemma}
\begin{proof}
As in \cite[Lemma 3.6]{CG}, choose a smooth ambient cutoff $\chi_R$
equal to one on $B_R(0)$ and zero outside $B_{2R}(0)$, with
$|\nabla\chi_R|\le C/R$. The assumed decay of $f$, $|A|=O(r^{-n})$,
and Euclidean volume growth give
\[
\|f\|_{\mathrm{en}}<\infty,\qquad
\|(1-\chi_R)f\|_{\mathrm{en}}\longrightarrow0.
\]
Properness gives $\chi_Rf\in C_c^\infty(\Sigma)$. 
Since $\chi_Rf\to f$, $f\in\BN$. The boundary
does not affect this argument.
\end{proof}

\begin{lemma}\label{lem:equality}
Fix $\psi\in C_c^\infty(\Sigma)$ with $\Q(\psi,\psi)<0$.
If $f\in\BN$ and $\Q(f,\psi)=0$, then $\Q(f,f)\ge0$.
Equality holds if and only if
\begin{equation}\label{eq:neumann-kernel}
Lf=0\text{ in }\Sigma^\circ,\qquad \partial_\eta f=0\text{ on }B.
\end{equation}
\end{lemma}
\begin{proof}
Set $W=\{h\in\BN:\Q(h,\psi)=0\}$. Suppose that some $f\in W$
satisfies $\Q(f,f)<0$. Then $f$ and $\psi$ are linearly independent:
if $af+b\psi=0$ in $\BN$, pairing with $\psi$ gives
$b\Q(\psi,\psi)=0$, hence $b=0$; since $f\ne0$, also $a=0$.
This is a two-dimensional negative definite subspace of $\BN$.
Density of $C_c^\infty(\Sigma)$ in $\BN$ implies that Q
has Neumann index 1 with respect to functions in $\BN$, a contradiction.
Thus $\Q\ge0$ on $W$.

Now let $f\in W$ and $\Q(f,f)=0$. For every $h\in W$ and $t\in\R$,
\[
0\le\Q(f+th,f+th)=2t\Q(f,h)+t^2\Q(h,h),
\]
so $\Q(f,h)=0$. For any $g\in\BN$, the function
\[
h=g-\frac{\Q(g,\psi)}{\Q(\psi,\psi)}\psi
\]
belongs to $W$, and consequently $\Q(f,g)=0$.
Let $\rho$ be the reflection on the double $M$, and extend $f$ evenly
to $\widehat f\in H^1_{\mathrm{loc}}(M)$. For every
$\phi\in C_c^\infty(M)$, reflection gives
\[
\int_M\bigl(\langle\nabla\widehat f,\nabla\phi\rangle
-|A_M|^2\widehat f\phi\bigr)
=\Q\bigl(f,(\phi+\phi\circ\rho)|_\Sigma\bigr)=0,
\]
where $A_M$ is the second fundamental form of $M$.
Thus $\widehat f$ weakly solves $(\Delta_M+|A_M|^2)\widehat f=0$.
By standard elliptic regularity, $\widehat f$ is smooth.
Restricting to $\Sigma$ gives $Lf=0$, and evenness gives
$\partial_\eta f=0$ on $B$.
Conversely, the two equations in
\eqref{eq:neumann-kernel} imply $\Q(f,g)=0$ for every compactly
supported smooth $g$, and density gives $\Q(f,f)=0$.
\end{proof}

\begin{lemma}[Boundary unique continuation]\label{lem:boundary-uc}
Let $V$ be a smooth function on $\Sigma$ and admit a smooth extension to its
double. If $w\in C^\infty(\Sigma)$ satisfies
\[
(\Delta+V)w=0\quad\text{in }\Sigma^\circ,\qquad
w=\partial_\eta w=0\quad\text{on a nonempty open subset of }B,
\]
then $w\equiv0$. In particular this holds for $V=0$ and $V=|A|^2$.
\end{lemma}
\begin{proof}
Let $U\subset B$ be the nonempty open set on which both Cauchy data vanish, and fix $p\in U$. On the smooth double $M$,
choose boundary normal coordinates $(y,t)$ in an open neighborhood
of $p$, with $B=\{t=0\}$ and $\Sigma=\{t\ge0\}$, so
\[
g=\dd t^2+g_{ij}(y,t)\,\dd y^i\dd y^j.
\]
Set $G=\det(g_{ij})$ then
\[
\begin{aligned}
\Delta_gw
&=\frac1{\sqrt G}\partial_t(\sqrt G\,\partial_tw)
+\frac1{\sqrt G}\partial_{y^i}
  (\sqrt G\,g^{ij}\partial_{y^j}w)\\
&=\partial_t^2w+a(y,t)\partial_tw+L_tw,
\end{aligned}
\]
where
\[
a=\partial_t\log\sqrt G,\qquad
L_tw=\frac1{\sqrt G}\partial_{y^i}
  (\sqrt G\,g^{ij}\partial_{y^j}w).
\]
The doubled metric and the smooth extension of $V$ are smooth, and hence
the equation becomes
\[
\partial_t^2w+a(y,t)\partial_tw+L_tw+V(y,t)w=0.
\]
Since $\eta=-\partial_t$
on $B$, the Cauchy data give $w(y,0)=\partial_tw(y,0)=0$.
Consequently all their tangential derivatives vanish as well.

Suppose $\partial_t^jw(y,0)=0$ for $0\le j\le r+1$.
Differentiate the equation $r$ times in $t$ and restrict to $t=0$.
Apart from $\partial_t^{r+2}w$, every term involves a tangential
derivative of $\partial_t^jw$ with $j\le r+1$, and hence vanishes.
Thus $\partial_t^{r+2}w(y,0)=0$. Induction gives
\[
\partial_y^\alpha\partial_t^jw(y,0)=0
\qquad\text{for all }\alpha\text{ and }j\ge0.
\]
Extend $w$ by zero to $t<0$. The vanishing of all boundary derivatives
makes this extension smooth; it solves the extended equation on the
whole neighborhood, including $t=0$, and vanishes to infinite order
at the interior point $p$ of $M$. Strong unique continuation
\cite{Aronszajn} therefore makes it zero near $p$. Unique continuation
on the connected interior of $\Sigma$ gives $w\equiv0$ on $\Sigma$.
\end{proof}

\section{Orthogonal ends}\label{sec:orthogonal}

Assume $\Pi\perp P$. We first recall the end estimates from
\cite[Section 3.1, Lemmas 3.3--3.4]{CG}. An end of $M$ is the graph
of a function $u_\alpha$ over $\Pi$, with
\[
Du_\alpha=O(r^{1-n}),\qquad |A|=O(r^{-n}),\qquad r=|x|.
\]
If $n_\Pi$ is a unit normal to $\Pi$, the exact graph formula gives
\[
\nu=\sigma_\alpha\frac{n_\Pi-Du_\alpha}{\sqrt{1+|Du_\alpha|^2}}
=\sigma_\alpha(n_\Pi-Du_\alpha)+O(r^{2-2n}),
\qquad \sigma_\alpha\in\{\pm1\}.
\]
Here $e_N\in\Pi$, so the translation Jacobi field
$v:=\ip{\nu}{e_N}$ satisfies
\begin{equation}\label{eq:translation-decay}
v=-\sigma_\alpha\frac{\ip{Du_\alpha}{e_N}}{\sqrt{1+|Du_\alpha|^2}}
=O(r^{1-n}),\qquad
\nabla v=-A(e_N^\top)=O(r^{-n}).
\end{equation}
Moreover, $Lv=0$ since translation preserves minimality, and the free boundary condition gives $v|_B=0$.

\begin{proposition}[Orthogonal ends]\label{prop:orthogonal}
If $\Pi\perp P$ and $I_N=1$, then $I_D=0$ and the double is a catenoid.
\end{proposition}
\begin{proof}
The decay in \eqref{eq:translation-decay} gives $v\in\BN$ by
Lemma~\ref{lem:energy-membership}.
Since $Lv=0$ and $v|_B=0$, integration by parts and approximation
in $\BN$ by compactly supported functions with zero boundary
value give $\Q(v,v)=0$.

Suppose, towards a contradiction, that $I_D\ge1$. Choose
$w\in C_c^\infty(\Sigma^\circ)$ with
$\Q(w,w)<0$. Since $Lv=0$ and $w$ is supported away from the boundary,
\[
\Q(v,w)=-\int_\Sigma wLv+\int_B w\partial_\eta v=0.
\]
By Lemma~\ref{lem:equality},
the hypotheses $v\in\BN$, $\Q(v,w)=0$, and $\Q(v,v)=0$ give
$\partial_\eta v=0$ on $B$. Together with $v|_B=0$, Lemma~\ref{lem:boundary-uc}
forces $v\equiv0$. However $v$ is nonzero. Otherwise $e_N$ is tangent everywhere to the complete double, so its flow makes $M$ translation-invariant in the $e_N$ direction. Finite total curvature then forces $A=0$, contradicting index one.

Hence $I_D=0$, \eqref{eq:parity} gives $\Ind(M)=1$, and \cite[Theorem 1.1]{CG} applies, forcing $M$ to be a catenoid.
\end{proof}

\section{Parallel ends: the relative harmonic-form space}\label{sec:relative}

Assume henceforth $\Pi\parallel P$. On each of the finitely many
regular ends of the double $M$, $|\ip{\nu}{e_N}|\to1$.
Consequently there is a compact set $K\subset M$ such that
$|\ip{\nu}{e_N}|>1/2$ on $M\setminus K$. The free boundary
condition gives $\ip{\nu}{e_N}=0$ on $B$, so $(M\setminus K)\cap P=\emptyset$ and $B\subset K$.
Since $B=M\cap P$ is closed in $M$, it is compact. 
Every end of $\Sigma$ is a full exterior graph over $P$. 
Let the number of the ends be $k\ge1$ and their heights be $h_\alpha\ge0$. 
At least one $h_\alpha>0$: otherwise the harmonic height $z$, zero on $B$ and tending to zero at infinity, would vanish by the maximum principle, contradicting $\partial_\eta z=-1$. 
For all sufficiently large $R$, set
\[
\Sigma_R=\Sigma\cap\overline{B_R(0)},\qquad
\Gamma_R=\Sigma\cap\partial B_R(0).
\]
Regularity of the ends implies that the large spheres meet $\Sigma$
transversely.

We construct only the exact harmonic forms needed below. Write
$\iota:B\hookrightarrow\Sigma$ for the inclusion.

\begin{proposition}\label{prop:dimension}
There is a $(k+n)$-dimensional space $\EE$ of forms $\omega=\dd F$,
where $F\in C^\infty(\Sigma)$ is harmonic, and
$F|_B=0$, containing $\dd z$.
Every $\omega\in\EE$ satisfies
\begin{equation}\label{eq:relative}
\iota^*\omega=0,\qquad
\nabla^\ell(\omega^\sharp-h_\alpha q^\top)=O(r^{1-n-\ell})
\quad\text{on }E_\alpha
\end{equation}
for a unique $q\in P$ and every $\ell\ge0$. Set $I(\omega):=q$.
\end{proposition}
\begin{proof}
For $f\in C_c^\infty(\Sigma)$, take $R$ sufficiently
large that $\operatorname{supp}f\subset\Sigma_R\setminus\Gamma_R$ and solve
\[
\Delta u_R=f\quad\text{in }\Sigma_R,\qquad
u_R=0\quad\text{on }\partial\Sigma_R=B\sqcup\Gamma_R,
\]
The barrier in \cite[Lemma 5.2]{CG} $w=(1+|x|^2)^{(2-n)/2}$ satisfies
\[
\Delta w=-n(n-2)(1+|x|^2)^{-(n+2)/2}(1+|x^\perp|^2)<0.
\]
Choose $C$ independent of $R$ with $-C\Delta w\ge|f|$.
Since $Cw\pm u_R\ge0$ on both boundary parts, the maximum
principle gives $|u_R|\le Cw$. Interior and boundary Schauder estimates let us
send $R\to \infty$ and give
\begin{equation}\label{eq:poisson}
\Delta u=f,\qquad u|_B=0,\qquad
\nabla^\ell u=O(r^{2-n-\ell}).
\end{equation}

Choose an orthonormal basis $p_1,\ldots,p_n$ of $P$ and smooth
cutoffs $\chi_\alpha\in C^\infty(\Sigma)$, equal to one on $E_\alpha$ and
vanishing near $B$ and
on other ends. Set
\[
G_\alpha=\chi_\alpha\quad(1\le\alpha\le k),\qquad
G_{k+j}=\sum_{\alpha=1}^k\chi_\alpha h_\alpha\ip{p_j}{x}
\quad(1\le j\le n).
\]
Constants and ambient coordinate functions are harmonic on $\Sigma$,
so each $\Delta G_a$ has compact support. Apply \eqref{eq:poisson}
with $f=-\Delta G_a$ and set $F_a=G_a+u_a$. These functions are
harmonic, zero on $B$, and satisfy
\begin{equation}\label{eq:basis-ends}
\begin{aligned}
F_\alpha&=\delta_{\alpha\beta}+O(r^{2-n}),&&1\le\alpha\le k,\\
F_{k+j}&=h_\beta\ip{p_j}{x}+O(r^{2-n}),&&1\le j\le n,
\end{aligned}
\qquad\text{on }E_\beta.
\end{equation}
Define $\EE=\Span\{\dd F_1,\ldots,\dd F_{k+n}\}$.
Its elements are closed and coclosed, and $F_a|_B=0$ gives the
relative condition $\iota^*\dd F_a=0$. Differentiating
\eqref{eq:basis-ends} proves \eqref{eq:relative}, with
$q=\sum_j b_jp_j$ for
$\omega=\sum_\alpha a_\alpha\dd F_\alpha+\sum_j b_j\dd F_{k+j}$.

To prove independence, suppose this $\omega$ vanishes. Set
\[
F=\sum_{\alpha=1}^k a_\alpha F_\alpha
+\sum_{j=1}^n b_jF_{k+j}.
\]
Then $\dd F=\omega=0$, so $F$ is constant on connected $\Sigma$.
Since $F|_B=0$ and $B\ne\varnothing$, $F\equiv0$. On $E_\beta$ this gives
\[
0=a_\beta+h_\beta\ip{q}{x}+O(r^{2-n}).
\]
If $q\ne0$, choose $h_\beta>0$ and let $x$ run along the graph
over $y=tq/|q|$. Then $\ip{q}{x}=t|q|$, a contradiction as
$t\to\infty$. Thus $q=0$, hence all $b_j=0$, and taking the limit
on each end gives all $a_\beta=0$. Therefore $\dim\EE=k+n$.
Finally, $z-\sum_\alpha h_\alpha F_\alpha$ is harmonic, zero on $B$,
and tends to zero on every end. The maximum principle gives
$z=\sum_\alpha h_\alpha F_\alpha$, so $\dd z\in\EE$.
\end{proof}

\begin{remark}
The half-catenoid has vanishing first Betti number for $n\ge3$.
We therefore retain only the exact forms constructed above and omit
any additional contribution involving Betti numbers from the dimension
lower bound.

Unlike the boundaryless construction, the common end-constant mode
is not lost on taking differentials: $\sum_\alpha F_\alpha$ is zero
on $B$ and tends to one on every end. This accounts for $k+n$
rather than $k+n-1$ exact forms.
\end{remark}

\section{The trace identity: exactly the additional boundary terms}\label{sec:trace}

We recall the admissible pairings of \cite[Section 6]{CG}.
Identify ambient vectors and covectors by the Euclidean metric, and let
$\star$ denote the Euclidean Hodge star. Fix a constant
$(N-4)$-form $\Omega$ such that the self-adjoint operator
\[
P_\Omega=\operatorname{Id}+\star(\Omega\wedge\cdot)
\quad\text{on }\Lambda^2\R^N
\]
is positive semidefinite. There exist constant ambient two-forms
$\Theta_1,\ldots,\Theta_r$ such that
\[
P_\Omega=\sum_{a=1}^r\Theta_a\otimes\Theta_a,
\qquad (\Theta_a\otimes\Theta_a)(\beta)
:=\ip{\beta}{\Theta_a}\Theta_a.
\]
By abusing notation, we call both $P_\Omega$ and $\{\Theta_1,\ldots,\Theta_r\}$ an admissible pairing set
associated with $\Omega$; see \cite[Section 6]{CG}.
For any $a,b\in\R^N$, \cite[Lemma 6.1]{CG} gives
\[
\ip{P_\Omega(a\wedge b)}{a\wedge b}=|a\wedge b|^2.
\]
For $\omega\in\EE$ with $I(\omega)=q$, put
\begin{equation}\label{eq:tests}
u_a:=\ip{\nu\wedge\omega^\sharp}{\Theta_a},\qquad
c_a:=\ip{e_N\wedge q}{\Theta_a},\qquad
Z:=\ip{x}{\nu}, \qquad
\widetilde u_a:=u_a-c_aZ.
\end{equation}
Thus $\sum_ac_a^2=\ip{P_\Omega (e_N\wedge q)}{e_N\wedge q}=|q|^2$. By the end calculation in \cite[Proposition 7.2]{CG}, $\widetilde u_a=O(r^{2-n})$ and $\nabla \widetilde u_a=O(r^{1-n})$. Lemma~\ref{lem:energy-membership} therefore gives $\widetilde u_a\in\BN$. 

The local identity in \cite[Lemma 7.1]{CG} applies unchanged:
\begin{equation}\label{eq:cg-local}
\sum_a(|\nabla u_a|^2-|A|^2u_a^2)\,\dd\mu
=\tfrac12\Delta|\omega|^2\,\dd\mu
+2(-1)^n\dd(\Omega_\Sigma\wedge\omega\wedge A\omega).
\end{equation}
Here $j:\Sigma\hookrightarrow\R^N$ is the inclusion and
$\Omega_\Sigma=j^*\Omega$.
Also, $A\omega$ is the one-form defined by
$(A\omega)(X)=\omega(AX)$. 

\begin{lemma}\label{lem:flux}
On $B$ one has
\begin{equation}\label{eq:flux-zero}
\partial_\eta|\omega|^2=0,\qquad
\iota^*(\Omega_\Sigma\wedge\omega\wedge A\omega)=0,
\qquad \partial_\eta Z=0.
\end{equation}
Moreover,
\begin{equation}\label{eq:Z-energy}
\Q_\infty(Z,Z)=(n-1)\int_\Sigma|\nabla z|^2>0.
\end{equation}
\end{lemma}
\begin{proof}
Recall that for a local orthonormal frame
$X_1,\ldots,X_n$ of $T\Sigma$, the codifferential is
$\delta=-\sum_{i=1}^n\iota_{X_i}\circ \nabla_{X_i}$.

Choose the frame adapted to the boundary, with $X_1,\ldots,X_{n-1}\in TB$ and $X_n=\eta$ on $B$. Since $\omega$ is coclosed,
\[
0=\delta\omega
=-\sum_{i=1}^{n-1}(\nabla_{X_i}\omega)(X_i)
-(\nabla_\eta\omega)(\eta).
\]
On $B$, the relative condition gives $\omega(X_i)=0$ for $i<n$,
and total geodesicity of $B\subset \Sigma$ gives $\nabla_{X_i}X_i=\nabla^B_{X_i}X_i\in TB$. Thus
\[
(\nabla_{X_i}\omega)(X_i)
=X_i(\omega(X_i))-\omega(\nabla_{X_i}X_i)=0,
\qquad i<n,
\]
and consequently $(\nabla_\eta\omega)(\eta)=0$.
Write $b:=\omega(\eta)$, so $\omega^\sharp=b\eta$ on $B$.
On $B$,
\[
\partial_\eta|\omega|^2
=2\langle\nabla_\eta\omega,\omega\rangle
=2\langle\nabla_\eta\omega,b\eta^\flat\rangle
=2b(\nabla_\eta\omega)(\eta)=0.
\]

Secondly, $\iota^*(\Omega_\Sigma\wedge\omega\wedge A\omega)=0$ because $\iota^*\omega=0$.

Finally, \eqref{eq:geometry} and $\nu\in P$ on $B$ give
\[
\partial_\eta Z=\ip{D_\eta x}{\nu}+\ip{x}{D_\eta \nu}=\ip{\eta}{\nu}-\ip{x}{A\eta}=-\kappa\ip{x}{\eta}=0.
\]

The dilation field satisfies $LZ=0$. By \cite[Lemma 4.1]{CG},
$Z$ has finite energy.
Choose $R$ sufficiently large that $B\subset B_R(0)$ and
$\partial B_R(0)$ meets all ends transversely, and set
\[
\Sigma_R=\Sigma\cap B_R(0),\qquad
\Gamma_R=\Sigma\cap\partial B_R(0),\qquad
\partial\Sigma_R=B\sqcup\Gamma_R.
\]
Let $\vartheta$ be the outward conormal on $\Gamma_R$.
As in \cite[Corollary 4.2]{CG}, the end expansions in
\cite[Lemmas 3.5 and 4.1]{CG} give
\[
Z\partial_\vartheta Z=(n-1)z\partial_\vartheta z+O(R^{-n})
\quad\text{on }\Gamma_R.
\]
Integrating by parts twice, first using $LZ=0$ and then $\Delta z=0$,
we obtain
\[
\begin{aligned}
\Q_R(Z,Z)
&=\int_{\Sigma_R}(|\nabla Z|^2-|A|^2Z^2)\\
&=-\int_{\Sigma_R}ZLZ
+\int_{\Gamma_R}Z\partial_\vartheta Z
+\int_B Z\partial_\eta Z\\
&=\int_{\Gamma_R}Z\partial_\vartheta Z\\
&=(n-1)\int_{\Gamma_R}z\partial_\vartheta z
+O(R^{-n})|\Gamma_R|\\
&=(n-1)\left(\int_{\Sigma_R}|\nabla z|^2
+\int_{\Sigma_R}z\Delta z-\int_B z\partial_\eta z\right)
+O(R^{-n})|\Gamma_R|\\
&=(n-1)\int_{\Sigma_R}|\nabla z|^2+O(R^{-1}).
\end{aligned}
\]
The boundary terms on $B$ vanish by $\partial_\eta Z=0$ and
$z|_B=0$, respectively; the error estimate uses
$|\Gamma_R|=O(R^{n-1})$ from \cite[Lemma 3.2]{CG}.
Letting $R\to\infty$ and using finiteness of the Dirichlet energy of $z$ completes the proof.
\end{proof}

\begin{proposition}\label{prop:trace}
For the functions in \eqref{eq:tests},
\begin{equation}\label{eq:corrected}
\sum_a\Q(\widetilde u_a,\widetilde u_a)=-|q|^2(n-1)\int_\Sigma|\nabla z|^2.
\end{equation}
\end{proposition}
\begin{proof}
Following \cite[Proposition 7.2]{CG}, we compute
\[
\sum_a\Q_R(\widetilde u_a,\widetilde u_a)+|q|^2\Q_R(Z,Z)
=\sum_a\Q_R(u_a,u_a)-2\sum_ac_a\Q_R(\widetilde u_a,Z).
\]
By integrating by parts in the second term, the only difference compared to \cite[Proposition 7.2]{CG} is some extra boundary terms
\begin{equation}\label{eq:physical}
\begin{aligned}
\mathfrak B={}&\int_B\left(\tfrac12\partial_\eta|\omega|^2
-2\sum_ac_a\widetilde u_a\partial_\eta Z\right)\,\dd S\\
&+2(-1)^n\int_B\iota^*(\Omega_\Sigma\wedge\omega\wedge A\omega).
\end{aligned}
\end{equation}
All three terms vanish by Lemma~\ref{lem:flux}. Thus the rest follows from the calculation of \cite[Proposition 7.2]{CG}. Passing to the limit and applying \eqref{eq:Z-energy} complete the proof.
\end{proof}

\section{Boundary rigidity and completion of the proof}\label{sec:rigidity}

We continue with the parallel-end case for the hypersurface $\Sigma$
of Theorem~\ref{thm:main}. Recall that $B=\partial\Sigma\ne\varnothing$,
$z=x_N$, and $\eta=-e_N$ on $B$. We consider exact harmonic forms
$\omega=\dd F$, where $F\in C^\infty(\Sigma)$ satisfies $\Delta F=0$ and $F|_B=0$, in particular
the forms in $\EE$ constructed in Proposition~\ref{prop:dimension}.
The Neumann
conditions in the following lemma imply $F\in\Span\{z\}$, as follows.

\begin{lemma}\label{lem:rigidity}
Let $F\in C^\infty(\Sigma)$ satisfy $\Delta F=0$ in $\Sigma^\circ$
and $F|_B=0$. For one fixed admissible pairing as above, define
\[
u_a=\ip{\nu\wedge\nabla F}{\Theta_a},\qquad 1\le a\le r.
\]
If $\partial_\eta u_a=0$ on $B$ for every $a=1,\ldots,r$, then
$F=\lambda z$ on $\Sigma$ for some constant $\lambda\in\R$;
equivalently, $\dd F=\lambda\dd z$.
\end{lemma}
\begin{proof}
Put $\omega=\dd F$ and $b=\omega(\eta)=\partial_\eta F$. Closedness and \eqref{eq:geometry} imply, for $X\in TB$,
\[
(\nabla_\eta\omega)(X)=(\nabla_X\omega)(\eta)=Xb.
\]
Coclosedness and total geodesicity give $(\nabla_\eta\omega)(\eta)=0$, as in Lemma~\ref{lem:flux}.
By \eqref{eq:geometry}, $$A(\eta,\omega^\sharp)\nu=\kappa\omega(\eta)\nu=\kappa b\nu$$
Therefore
\[
\nabla_\eta\omega^\sharp=\nabla^B b,
\qquad D_\eta\omega^\sharp=\nabla^B b+\kappa b\nu.
\]
Using $D_\eta\nu=-\kappa\eta$ and $\omega^\sharp=b\eta$ on $B$, we obtain
\begin{equation}\label{eq:normal-wedge}
\begin{aligned}
D_\eta(\nu\wedge\omega^\sharp)
&=(-\kappa\eta)\wedge(b\eta)
+\nu\wedge(\nabla^B b+\kappa b\nu)\\
&=\nu\wedge\nabla^B b.
\end{aligned}
\end{equation}
Taking pairing components and using decomposability gives
\[
0=\sum_a(\partial_\eta u_a)^2
=\ip{P_\Omega(\nu\wedge\nabla^B b)}{\nu\wedge\nabla^B b}
=|\nabla^B b|^2.
\]
Hence, for any $p\in B$, choose a connected open neighborhood $U\subset B$ of $p$; then $b=b_0$ on $U$ for some constant $b_0$. On $U$, the harmonic function $F+b_0z$ satisfies
\[
F+b_0z=0,\qquad \partial_\eta(F+b_0z)=b_0-b_0=0.
\]
Lemma~\ref{lem:boundary-uc} gives $F=-b_0z$ globally. 
\end{proof}

\begin{proof}[Proof of Theorem~\ref{thm:main}]
In $\R^3$, the index-one half-catenoid classification is by \cite[Corollary 24]{ABCS}.

In the orthogonal-end case the double is a catenoid by Proposition~\ref{prop:orthogonal}. 

Now consider the parallel-end case. If $k\ge2$, Proposition~\ref{prop:dimension} gives $\dim\EE=k+n\ge N+1$. Choose an $N$-dimensional subspace $V\subset\EE$ with
\[
V\cap\Span\{\dd z\}=\{0\}.
\]
Fix $\psi\in C_c^\infty(\Sigma)$ with $\Q(\psi,\psi)<0$. Define $T:V\to\Lambda^2\R^N$ by
\[
\ip{T(\omega)}{\Theta}
=\Q_\infty\bigl(\ip{\nu\wedge\omega^\sharp}{\Theta},\psi\bigr).
\]
This is well-defined since $\psi$ is compactly supported.
If $T$ has a nonzero kernel vector, take $P_\Omega=\operatorname{Id}$; otherwise $T(V)$ has dimension $N$, and \cite[Propositions 6.7--6.8]{CG} give a nonzero $\omega\in V$ and an admissible $P_\Omega$ with $P_\Omega T(\omega)=0$. Factoring the pairing yields
\[
0=\ip{P_\Omega T(\omega)}{T(\omega)}
=\sum_a\Q_\infty(u_a,\psi)^2.
\]
Thus $\Q_\infty(u_a,\psi)=0$ for every $a$.
Since $LZ=0$ and $\partial_\eta Z=0$, integration by parts against
$\psi$ gives $\Q_\infty(Z,\psi)=0$. Hence
$\Q(\widetilde u_a,\psi)=0$.
Lemma~\ref{lem:equality} gives $\Q(\widetilde u_a,\widetilde u_a)\ge0$, whereas
Proposition~\ref{prop:trace} gives
\[
0\le\sum_a\Q(\widetilde u_a,\widetilde u_a)
=-|q|^2(n-1)\int_\Sigma|\nabla z|^2\le0.
\]
By \eqref{eq:Z-energy}, $q=0$ and $\Q(\widetilde u_a,\widetilde u_a)=0$ for every $a$.
Thus $c_a=0$ and $\widetilde u_a=u_a$; the equality statement in
Lemma~\ref{lem:equality} yields $Lu_a=0$ and $\partial_\eta u_a=0$.
 Since $\omega\in\EE$, there is $F\in C^\infty(\Sigma)$ with $\omega=\dd F$, $\Delta F=0$, and $F|_B=0$. Lemma~\ref{lem:rigidity} gives $\omega\in\Span\{\dd z\}$, contradicting the choice of $V$.

Thus $k=1$. Since $B$ is compact, the upper end and its reflected lower end are disjoint outside a compact set, so the double has exactly two regular ends. Schoen's two-ended theorem \cite[Theorem 3]{Schoen} makes it a catenoid.

\end{proof}

\section{Free Boundary Minimal Hypersurfaces in a Quarter-Space}\label{sec:quarter}
Now we consider a different ambient space. Let
$W=\{x_1\ge0,\ x_2\ge0\}\subset\R^{n+1}$, $n\ge3$, with bounding
hyperplanes $P_i=\{x_i=0\}$ and reflections of $\R^{n+1}$ across $P_i$ being $\rho_i$, $i=1,2$.
A \emph{quarter minimal hypersurface} is a complete, connected,
two-sided, embedded minimal hypersurface $\Lambda\subset W$ of finite
total curvature, with $\partial\Lambda=B_1\cup B_2$, where
$B_i=\Lambda\cap P_i$, meeting each $P_i$ orthogonally.
We require $B_i\setminus E\ne\varnothing$, where $E=B_1\cap B_2$, $i=1,2$. It's a manifold with corners, with local model
$[0,\infty)^2\times\R^{n-2}$ at $E$.
Let $\nu$ be a unit normal to
$\Lambda$; the outward conormal to $B_i\setminus E$ is $\eta_i=-e_i$.

Reflecting gives $\Sigma=\Lambda\cup\rho_1(\Lambda)$
, which is a complete, connected,
two-sided, embedded free
boundary minimal hypersurface with nonempty boundary on $P_2$, which is smooth by standard elliptic theory.
By Section~\ref{sec:forms}, the ends of $M$ have a common asymptotic
hyperplane direction $\Pi$. Since $\Pi^\perp$ is invariant under both
reflections, $\Pi$ is orthogonal to at least one $P_i$.
Relabeling the planes if necessary, assume $\Pi\perp P_1$, so
$e_1\in\Pi$, and set $v:=\ip{\nu}{e_1}$ on $\Lambda$.

Write $\Q_\Lambda(f,g)=\int_\Lambda
(\ip{\nabla f}{\nabla g}-|A|^2fg)$. Define $I_{NN}(\Lambda)$ and
$I_{DN}(\Lambda)$ as the maximal dimensions of subspaces on which
$\Q_\Lambda$ is negative definite in $C_c^\infty(\Lambda)$ and
$C_c^\infty(\Lambda\setminus B_1)$, respectively.
In the latter space functions are extended by zero to $B_1$.
The first letter refers to $B_1$ and the second to $B_2$. 
The Morse index of $\Lambda$ as a free boundary minimal surface equals $I_{NN}(\Lambda)$.
The even--odd
decomposition under $\rho_1$, as in
\cite[Lemmas 16, 18, 20]{ABCS}, gives
\begin{equation}\label{eq:quarter-parity}
I_{DN}(\Lambda)\le I_{NN}(\Lambda),\qquad
I_N(\Sigma)=I_{NN}(\Lambda)+I_{DN}(\Lambda).
\end{equation}

\paragraph{Example (Quarter-catenoids).}
Up to rigid motions and scaling, there are two types of quarter-catenoids.
Let $C\subset\R^N$ be a catenoid with axis $\R e_N$ and waist in
$\{x_N=0\}$. Cutting by two orthogonal hyperplanes containing the
axis gives $C\cap\{x_1\ge0,\ x_2\ge0\}$, which has two ends.
Cutting by the waist hyperplane and a hyperplane containing the axis
gives $C\cap\{x_1\ge0,\ x_N\ge0\}$, which has one end.
Both meet their bounding hyperplanes orthogonally. These are the only
two types, since every reflection symmetry hyperplane of a catenoid
either contains its axis or is its waist hyperplane.

\begin{theorem}\label{thm:quarter}
If a quarter minimal hypersurface $\Lambda$ has $I_{NN}(\Lambda)=1$,
then it is a quarter-catenoid.
\end{theorem}
\begin{proof}
In $\R^3$, the index-one quarter-catenoid classification is by Chen \cite[Proposition 3.1]{Chen}.

Now consider $n+1\geq 4$. Write $\Q=\Q_\Lambda$. The calculation in
\eqref{eq:translation-decay}, with $e_1$ in place of $e_N$, gives
$v=O(r^{1-n})$ and $\nabla v=O(r^{-n})$, and $Lv=0$.
The free boundary condition gives $v|_{B_1}=0$.
On $B_2\setminus E$, \eqref{eq:geometry} gives
$A\eta_2=\kappa_2\eta_2$, hence
$\partial_{\eta_2}v=-\ip{A\eta_2}{e_1}=0$.
Moreover $v\not\equiv0$: otherwise $e_1$ is tangent everywhere to
the complete double $M$, which is then translation-invariant in
this direction. Finite total curvature forces $A=0$, contradicting
$I_{NN}(\Lambda)=1$.

For $g\in C_c^\infty(\Lambda)$ with $g|_{B_1}=0$, integration by
parts on the manifold with corners \cite[Theorem 16.25]{Lee} gives
\[
\begin{aligned}
\Q(v,g)
&=\int_\Lambda\bigl(\ip{\nabla v}{\nabla g}-|A|^2vg\bigr)\\
&=-\int_\Lambda gLv
  +\int_{B_1\setminus E}g\,\partial_{\eta_1}v
  +\int_{B_2\setminus E}g\,\partial_{\eta_2}v\\
&=0.
\end{aligned}
\]
The interior term vanishes because $Lv=0$; the two boundary terms
vanish because $g|_{B_1}=0$ and $\partial_{\eta_2}v|_{B_2\setminus E}=0$,
respectively. 
By the same argument as in section~\ref{sec:forms}, $v$ can be approximated in the energy norm by compactly supported smooth functions vanishing on \(B_1\).
Hence \(Q(v,v)=0\).

It suffices to prove $I_{DN}(\Lambda)=0$. Suppose $I_{DN}(\Lambda)>0$, and choose
$w\in C_c^\infty(\Lambda\setminus B_1)$ with $\Q(w,w)<0$.
Then $\Q(v,w)=0$ since $w|_{B_1}=0$. The local boundary unique continuation argument
of Lemma~\ref{lem:boundary-uc}, applied away from $E$, implies
$\partial_{\eta_1}v\not\equiv0$ on $B_1\setminus E$.
Choose a cutoff $\chi\in C_c^\infty(B_1\setminus E)$ supported in one
boundary coordinate neighborhood with $\partial_{\eta_1}v\not\equiv0$
, and extend
$\chi\partial_{\eta_1}v$ to a function $h\in C_c^\infty(\Lambda)$
supported away from $B_2$. Such an extension is obtained in boundary
normal coordinates by setting
$h(y,t)=\chi(y)\partial_{\eta_1}v(y)\theta(t)$, where $\theta(0)=1$
and $\theta$ is a smooth cutoff supported in the coordinate neighborhood.
Integration by parts yields
\[
a:=\Q(v,h)=\int_{B_1}h\partial_{\eta_1}v
=\int_{B_1}\chi(\partial_{\eta_1}v)^2>0.
\]
Set $h_0=h-\Q(h,w)\Q(w,w)^{-1}w$. Then
$\Q(h_0,w)=0$ and $\Q(v,h_0)=a$, so, for sufficiently small $t>0$,
\[
\Q(v-th_0,w)=0,\qquad
\Q(v-th_0,v-th_0)=-2ta+t^2\Q(h_0,h_0)<0.
\]
Thus $w$ and $v-th_0$ are linearly independent and span a negative
definite subspace. Replacing $v$ by $v_R$ preserves this strict
negative definiteness for large $R$, giving two compactly supported
negative directions and contradicting $I_{NN}(\Lambda)=1$.
Hence $I_{DN}(\Lambda)=0$, and \eqref{eq:quarter-parity} gives
$I_N(\Sigma)=1$. By Theorem~\ref{thm:main}, applied with supporting
hyperplane $P_2$, $\Sigma$ is a half-catenoid. Consequently
$\Lambda$ is a quarter-catenoid.
\end{proof}
\begin{remark}
We do not use the argument of Section~\ref{sec:orthogonal}, although
it could also be adapted to this setting. Instead, the proof uses a
localized version of the perturbation underlying Tran's inequality
\cite{Tran} as stated in \cite[Theorem 2.3]{Chen}, as
the nonzero Dirichlet--Neumann Jacobi field $v$ on $\Lambda$ yields
an additional negative direction.
\end{remark}


\begin{thebibliography}{ABCS19}
\raggedright
\bibitem[CG26]{CG}
O.~Chodosh and M.~Gianocca,
\emph{Minimal hypersurfaces of Morse index one}, arXiv:2607.27444v1 (2026).
\url{https://arxiv.org/abs/2607.27444}.
\bibitem[Lu26]{Lu}
C.~Lu, \emph{Index-one capillary minimal hypersurfaces in $\R^n_+$},
arXiv:2610.24878v1 (2026).
\url{https://arxiv.org/abs/2610.24878}.
\bibitem[ABCS19]{ABCS}
L.~Ambrozio, R.~Buzano, A.~Carlotto, and B.~Sharp,
\emph{Bubbling analysis and geometric convergence results for free boundary minimal surfaces}, J. \'Ec. polytech. Math. \textbf{6} (2019), 621--664.
\url{https://doi.org/10.5802/jep.102}.
\bibitem[Che23]{Chen}
S.~Chen, \emph{On the index of minimal surfaces with free boundary in a
half-space}, J. Geom. Anal. \textbf{33} (2023), Paper No. 46.
\url{https://doi.org/10.1007/s12220-022-01106-8}.
\bibitem[And84]{Anderson}
M.~Anderson, \emph{The compactification of a minimal submanifold in
Euclidean space by the Gauss map}, manuscript (1984).
\url{https://www.math.stonybrook.edu/~anderson/compactif.pdf}.
\bibitem[Sch83]{Schoen}
R.~Schoen, \emph{Uniqueness, symmetry, and embeddedness of minimal surfaces},
J. Differential Geom. \textbf{18} (1983), 791--809.
\url{https://doi.org/10.4310/jdg/1214438183}.
\bibitem[TZ09]{TZ}
L.-F.~Tam and D.~Zhou, \emph{Stability properties for the higher dimensional catenoid in $\R^{n+1}$}, Proc. Amer. Math. Soc. \textbf{137} (2009), 3451--3461.
\url{https://arxiv.org/abs/0708.3310}.
\bibitem[Aro57]{Aronszajn}
N.~Aronszajn, \emph{A unique continuation theorem for solutions of elliptic partial differential equations or inequalities of second order}, J. Math. Pures Appl. (9) \textbf{36} (1957), 235--249.
\bibitem[GT01]{GT}
D.~Gilbarg and N.~S.~Trudinger, \emph{Elliptic Partial Differential Equations of Second Order}, Classics in Mathematics, Springer, 2001.
\bibitem[Tra20]{Tran}
H.~Tran, \emph{Index characterization for free boundary minimal surfaces},
Comm. Anal. Geom. \textbf{28} (2020), no.~1, 189--222.
\url{https://doi.org/10.4310/CAG.2020.v28.n1.a6}.
\bibitem[Lee13]{Lee}
J.~M.~Lee, \emph{Introduction to Smooth Manifolds}, 2nd ed.,
Graduate Texts in Mathematics, vol.~218, Springer, New York, 2013.
\url{https://doi.org/10.1007/978-1-4419-9982-5}.
\end{thebibliography}
\end{document}